\documentclass{amsart}
\usepackage{verbatim,amssymb,amsmath,amscd,latexsym,amsbsy,mathrsfs}
\usepackage{graphicx} \input{xy}
\xyoption{all}

\newtheorem{thm}{Theorem}[section] \newtheorem{cor}[thm]{Corollary}

\DeclareMathAlphabet{\mathpzc}{OT1}{pzc}{m}{it}
 \DeclareMathOperator{\Hom}{Hom}
\DeclareMathOperator*{\End}{End}

\DeclareMathOperator*{\into}{\hookrightarrow}
 
 \newcommand{\QQ}{\mathbb{Q}}
 
 \newcommand{\PP}{\mathbb{P}}

\newcommand{\onto}{\twoheadrightarrow}

 \newcommand {\C} {{\mathbb C}}
 \newcommand {\Z} {{\mathbb Z}}
\newcommand {\Q} {{\mathbb Q}} 
    \newcommand {\dt} {{\bullet}} \newcommand {\G}
{{\mathbb G}}  \newcommand {\OO}
{{\mathcal O}} \newcommand {\A} {\mathbb{A}} \newcommand {\I}
{\mathcal{I}}

\newtheorem{lemma}[thm]{Lemma}
\newtheorem{prop}[thm]{Proposition}

\begin{document}
\title{ Beilinson-Tate cycles on semiabelian varieties} \author{ Donu
  Arapura } \thanks{First author partially supported by NSF}
\address{Department of Mathematics\\
  Purdue University\\
  West Lafayette, IN 47907\\
  U.S.A.}  \author{ Manish Kumar }
\address{Department of Mathematics\\
  Michigan State University\\
  East Lansing, MI-48824\\
  U.S.A.}
\maketitle

In \cite{AK}, it was shown that the Beilinson-Hodge conjecture holds
for a product of smooth curves and semiabelian varieties. In this
companion paper, we establish the ``Tate'' version that given such a
variety over (say) a number field, the Galois invariant cycles of
highest weight on \'etale cohomology come from motivic cohomology.
We also give criteria for similar statements to hold for lower weight cycles 
in both the Hodge and Tate cases.
More precisely, given a smooth not necessarily proper variety $U$ over
a field $k$, we have the so called regulator or cycle maps
\begin{itemize}
\item $CH^i(U,j)\otimes \Q\to Hom_{MHS}(\Q(0), H^{2i-j}(U,\Q(i)))$
  ($k=\C$, Hodge-version)
\item $CH^i(U,j)\otimes \Q_l\to H_{et}^{2i-j}(U\times_k
  \bar{k},\Q_l(i))^{G_k}$ ($k$ finitely generated over $\Q$,
  Tate-version)
\end{itemize}
from Bloch's higher Chow groups \cite{bloch}. Here $G_k$ is the
absolute Galois group $Gal(\bar k/k)$. Following Asakura and Saito
\cite{as}, we refer to the surjectivity of the first map when $i=j$ as
the Beilinson-Hodge conjecture, and surjectivity of the second when
$i=j$ as Beilinson-Tate.  One may consider, as Beilinson originally
had, the surjectivity question in either case for all pairs $(i,j)$,
but then things becomes a bit more subtle.  Surjectivity is known to
fail for some pairs $i\not=j$ when the field of definition is
transcendental, but surjectivity is expected over number fields,
c.f. \cite{jannsen}.  We will refer to this form as the ``strong''
Beilinson's Hodge and Tate conjecture. 

% Recall that a group variety $A$ over a field $k$ is called a semiabelian 
% variety if $A$ is an extension of an abelian variety $A_0$ by a torus $T$. 
% More precisely, there is an exact sequence
% $$\xymatrix{ 1\ar[r] &  T  \ar[r]^{\phi} & A \ar[r]^{\psi}& A_0\ar[r] & 0}$$
% where $\phi$ and $\psi$ are morphisms of algebraic group varieties and
% $T(\bar{k})=\G_m^n(\bar{k})$ for some $n$.
 
Our goal here is to extend the results of \cite{AK} to include strong 
version of Beilinson's Hodge conjecture. In proposition \ref{strong} and \ref{main2}
it is shown that the strong version of Beilinson Hodge conjecture holds for product of 
curves and semiabelian varieties if the classical Hodge conjecture for their 
smooth compactification is true.
The conjecture is shown to hold even for 
varieties dominated by product of curves (corollary \ref{maincor}).

We also show the Beilinson-Tate conjecture for products of curves,
semiabelian varieties and more generally for varieties dominated by
products of curves.  We recall that a semiabelian variety
over a field is an extension of abelian variety by a (possibly nonsplit) torus.
Also the strong version of Beilinson-Tate conjecture
holds in these cases if the Tate conjecture holds for their smooth 
compactification. The proof of the main result in \cite{AK} was based on 
analysing invariants under the Mumford-Tate group. A similar method is used
for Beilinson-Tate conjectures, but the Galois group plays the role of the
Mumford-Tate group in this case.

\section{Cycle map on higher Chow groups}\label{cycle-map}

The regulator map was originally described using Chern classes in
higher $K$-theory \cite{beilinson1,beilinson,gillet,soule}. Bloch
\cite{bloch1, bloch} later recast this as a more explicit cycle map on
his higher Chow group. This group is a fundamental object. It can be
interpreted as motivic cohomology \cite{levine, mvw}, and rationally
it coincides with the associated graded pieces of $K$-theory with
respect to the $\gamma$-filtration.  From the point of view of
motives, the regulator can be understood as coming from an appropriate
realization functor \cite{huber}.

We felt it useful to describe Bloch's construction in some detail, in
order to make certain properties clear. Let $U$ be a smooth variety
defined over a field $k$.  In Bloch's original definition, he viewed
$\A_k^n$ as a simplex and defined $CH^q(U,n)$ as the homology of the
complex $z^q(U,n)\subset Z^q(U\times \A^n)$ of codimension $q$ cycles
meeting the simplicial faces properly.  Here it is more convenient to
use the cubical approach \cite{totaro}, where we view $\A^n_k$ as a
cube with faces obtained by setting some of the coordinates to $0$ or
$1$. The codimension one faces can be described as images of the face
maps $\partial_i^{\epsilon}:\A^{n-1}\to \A^n$, $\epsilon\in\{0,1\}$
given by
$$(t_1,\ldots t_{n-1})\mapsto (t_1,\ldots t_{i-1},\epsilon,t_i,\ldots)$$
Let $C^q(U,n)\subset Z^q(U\times \A^n)$ denote the group generated by
codimension $q$ cycles meeting faces properly.  Then we have a complex
$$\ldots \to C^q(U,n)\to C^q(U,n-1)\to\ldots$$
with differential
\begin{equation}
  \label{eq:cubicaldiff}
  \partial = \sum_i (-1)^i (\partial_i^{0*}-\partial_i^{1*})  
\end{equation}
So as to use cohomology indexing, put $ C^q(U,n)$ in degree $-n$.  In
order to get the right cohomology, we have to divide out by the
subcomplex of degenerate cycles $D^q(U,n)$ spanned by pullbacks of
cycles from $U\times \A^{n-1}$ along coordinate projections $p_i$. Note
that the condition is vacuous for $n=0$, so $D^q(U,0)=0$.  Put $\bar
C^q(U,n) = C^q(U,n)/D^q(U,n)$.  Then $CH^q(U,n) = H^{-n}(\bar
C^q(U,\dt))$.

Suppose now that the ground field $k=\C$.  Let $S^\dt(U)$ be the
singular cochain complex.  Given Zariski closed $Z\subset U$, let
$S^\dt_Z(U) = \ker[S^\dt(U)\to S^\dt(U-Z)]$. This computes the
relative cohomology $H_Z^*(U) := H^*(U,U-Z) = H^*(S_Z(U))$. For a
cycle $Z$, we take $S_Z=S_{supp\, Z}$.  We define a double complex
$$S^{ab}=S^{ab}(U)=\varinjlim_{Z\in C^q(U,-a)} S_Z^b(U\times \A^{-a})$$
in the second quadrant with vertical differential given by simplicial
coboundary and horizontal differential $\partial $ given by
(\ref{eq:cubicaldiff}). To avoid convergence issues, we should
truncate this below a certain $a\ll 0$ as in \cite{bloch1}. However,
to avoid excessive notation we leave this step implicit. We obtain a
spectral sequence
$$E_1^{-a,b}(U)= H^b(\varinjlim S^\dt_Z(U\times \A^{a})) \cong 
\varinjlim H_Z^b(U\times \A^{a})\Rightarrow H^{b-a}(Tot(S^{\dt\dt}))$$
As above, we have a subcomplex of degenerate cohomology
$${}_DE_1^{-a,b}(U) = \sum p_i^*H_Z^b(U\times \A^{a-1})$$
and we set $\bar E_1^{ab} = E_1^{ab}/{}_DE_1^{ab}$ to the quotient
spectral sequence. It's not quite clear what $\bar E_1^{ab} $
converges to, but we don't really care. If we drop the support
condition by setting
$${}'S^{ab}=S^b(U\times \A^{-a})$$
and define ${}'E_1^{ab}, {}_D'E_1^{ab}, {}'\bar E_1^{ab}$ exactly as
before using ${}'S^{ab}$, then
$${}'\bar E_1^{ab} =
\begin{cases}
  H^b(U) &\text{if $a=0$}\\
  0 & otherwise
\end{cases}
$$
Since we have a map of spectral sequences $\bar E_1^{ab}\to {}'\bar
E_1^{ab}$, we conclude the abutment of the first spectral sequence
$E^{n}$ maps to the abutment of the second which is just $H^{n}(U)$.

% We have ${}_SE_1^{ab} \cong H^b(U)$ and $d:{}_SE_1^{ab}\to
% {}_SE_1^{a+1,b}$ is given by
% $$d=
% \begin{cases}
%   id& \text{$a$ even}\\
%   0 &\text{otherwise}
% \end{cases}
% $$
% Thus the spectral sequences collapses to
% $$H^{i}(Tot(S^{\dt\dt})) \cong H^i(U)$$
% In order to define a cycle map, it is useful to refine this by
% imposing a support condition.  Define
% $$T^{ab}= \varinjlim_{Z\in z^q(U,-a)} S_Z^b(U\times \Delta_{-a})$$
% This maps to $S^{\dt\dt}$.  We have a spectral sequence
% $${}_TE_1^{ab}= H^b(S_Z^\dt(U\times \Delta_{-a})) =  \varinjlim_Z H_Z^b(U\times \Delta_{-a})
% \Rightarrow H^{a+b}(Tot(T^{\dt\dt}))$$ as above, but it is more
% subtle.

We have semipurity that says $H_Z^i(U) =0$ for $i<2codim(Z)$.
Therefore $\bar E^{-a,b}_1$ vanishes below the line $b=2q$. Thus $\bar
E_2^{-a,2q}$ maps onto $\bar E_\infty^{-a,2q}\subseteq \bar E^{2q-a}$
which in turn maps to $H^{2q-a}(U)$.  This map can be described more
explicitly using the following diagram
$$
\xymatrix{
  \varinjlim H^{2q}_Z(U\times \A^{a}, U\times \partial\A^{a})\ar@{>>}[r]\ar[d] & \ker[d_1:\bar E_1^{-a,2q}\to \bar E_1^{-a+1,2q}]\ar[d] \\
  H^{2q}(U\times \A^{a}, U\times \partial\A^{a})\ar[r]^{\sim} &
  H^{2q-a}(U) }
$$
where $\partial \A^{a}$ is the union of codimension one faces. The top
line comes from the long exact sequence for the pair $(U\times
\A^{a},U\times \partial \A^{a})$. The map labelled $\sim$ is an
isomorphism by the K\"unneth formula. This description shows that the
map is compatible with mixed Hodge structures, provided we use
$H^{2q-a}(U)(q)$ at the last step.

We have a map of complexes
$$ c:C^q(U,\dt)\to E_1^{\dt,2q}$$
given by sending a cycle to its fundamental class, which induces
$$ c:\bar C^q(U,\dt)\to \bar E_1^{\dt,2q}$$
This induces a map
$$reg:CH^q(U,n)\to \bar E_2^{-n,2q}\to H^{2q-n}(U)$$
By the previous remark, the image lies in the group of cycles of type
$(0,0)$ in $ H^{2q-n}(U,\Z(q))$.

\subsection{Properties}
We now show that $reg$ is compatible with pushforwards, pullbacks, and
products. Also that
$$reg:CH^1(U,1)\cong \OO(U)^*\to H^1(U,\Z(1))$$
is the connecting map associated to the exponential sequence.

Given a proper map $f:U\to V$, and cycle $\alpha\in C^q(U\times \A^n)$
we can push it forward in the usual way \cite{fulton} to get a cycle
in $C^{q+c}(U\times \A^n)$ where $c=\dim V-\dim U$.  This can be seen
to define map of complexes $f_*:\bar C^q(U,\dt)\to \bar
C^{q+c}(V,\dt)$ \cite{bloch}. Note we also have a pushforward on the
spectral sequences
$$f_*:\bar E_1^{ab}(U)\to \bar E_1^{a,b+2c}(V),$$
which can be checked to be compatible with the cycle map $c$ defined
above (cf \cite[\S 19.1]{fulton}). Thus $reg$ is compatible with
pushforward.

Bloch originally described the groups $CH^q(U,n)$ as the cohomology of the
complex $z^q(U,\dt)\subset Z^q(U\times \A^\dt)$.  In this setting it
is no longer necessary to divide out the degenerate cycles. There is a
spectral sequence ${}_\dt E_1^{ab}(U)$ analogous to $E_1^{ab}(U)$, and
$reg$ can be described in terms of a map $z^q(U,\dt)\to {}_\dt
E_1^{\dt,2q2}(U)$ as above \cite{bloch1}.  If $V$ is affine then
Levine \cite[part I chap II,\S 3.5]{levine} showed that $z^q(V,\dt)$
is quasiisomorphic to a subcomplex $z^q(V,\dt)_f$ consisting of cycles
whose scheme theoretic pullback defines a cycle in $z^q(U,\dt)$. One
can see that one has a commutative diagram
$$
\xymatrix{
  z^q(U,\dt)_f\ar[r]\ar[d] & {}_\dt E_1^{\dt,2q}(U)\ar[d] \\
  z^q(V,\dt)\ar[r] & {}_\dt E_1^{\dt,2q}(V) }
$$
This shows compatibility with $reg$ and pullbacks when $V$ is affine.
The general case can be reduced to this using the Mayer-Vietoris for
$H^*$ and $CH^*$ \cite[part I, chap II]{levine} and the $5$-lemma.

Returning to the cubic viewpoint, the external product
$$CH^q(U,n)\times CH^p(V,m)\to CH^{q+p}(U\times V, n+m)$$
is induced by the map of complexes
$$C^q(U,\dt)\otimes C^p(V,\dt)\to C^{q+p}(U\times V,\dt)$$
given by $Z\otimes W\mapsto Z\times W$. It is fairly clear that
$c(Z\times W)=c(Z)\cup c(W)$, and therefore $reg$ preserves external
products. The cup product which is the external product followed by
the pullback under the diagonal map $\Delta:U\to U\times U$ must
therefore also be preserved.

By Bloch \cite{bloch}, $CH^1(U,1)= \OO(U)^*$.

\begin{prop}
  $reg:CH^1(U,1)\to H^1(U,\Z(1))$ coincides with the connecting map
  associated with the exponential sequence, at least up to sign.
\end{prop}

Here is a sketch.  Set $\Delta=\A^1$. Let $\I$ denote the ideal sheaf
of $U\times \partial \Delta= U\times \{0,1\}$ in $U\times \Delta$.
Let $j:\Delta-\{0,1\}\to \Delta$ be the inclusion.  Consider the
diagram
$$
\xymatrix{
  & 0\ar[d] & 0\ar[d] & 0\ar[d] &  \\
  0\ar[r] & j_!\Z(1)\ar[r]\ar[d] & \Z_{U\times \Delta}(1)\ar[r]\ar[d] & \Z_{U\times \{0,1\}}(1)\ar[r]\ar[d] & 0 \\
  0\ar[r] & \I\ar[r]\ar[d]^{\exp} & \OO_{U\times \Delta}\ar[r]\ar[d]^{\exp} & \OO_{U\times \{0,1\}}\ar[r]\ar[d]^{\exp} & 0 \\
  1\ar[r] & (1+\I)^*\ar[r]\ar[d] & \OO_{U\times \Delta}^*\ar[r]^{r}\ar[d] & \OO_{U\times \{0,1\}}^*\ar[r]\ar[d] & 1 \\
  & 1 & 1 & 1 & }
$$
where $(1+\I)^*$ is defined as the kernel of $r$. From this we obtain
a diagram
$$
\xymatrix{
  & \OO(U\times\{0,1\})^*\ar[r]\ar[d] & H^1((1+\I)^*)\ar[d]^{c_1} \\
  H^1(U\times \Delta)\ar[r] & H^1(U\times\{0,1\},\Z)\ar[r]\ar[d] & H^2(U\times\Delta,U\times\{0,1\}) \\
  & H^1(U\times \{0\},\Z)\ar[ru]^{\sim} & }
$$
The cohomology group $H^1(U,(1+\I)^*)$ is a relative Picard group,
which can be described in terms of certain divisor classes on $U\times
\Delta$ \cite[\S 6]{bloch}.  In particular, a divisor $Z$ in $\ker
[C^1(U,1)\to C^1(U,0)]$, locally defined by $f_i=0$, gives an element
$\OO(Z)_{rel}= \{f_i/f_j\}$ of $H^1((1+\I)^*)$.  By refining the usual
arguments, one can see that the fundamental class of $Z$ in
$H^2(U\times \Delta,U\times\{0,1\})$ coincides with $\pm
c_1(\OO(Z)_{rel})$.

\subsection{\'Etale version}

For a scheme $U$ over a field $k$ let $\bar{U}$ denote $U\times_k
\bar{k}$. Choose a prime $l\not= char(k)$.  We can then define a
double complex
$$S^{a\dt}(U) = \varinjlim_{Z\in C^q(U,-a)} \Gamma_Z(\bar{U}\times
\A^{-a}, I^{\dt}_{-a}(q))$$ as above, where $I^{\dt}_n(q)$ is an
injective resolution of the \'etale sheaf $\mu_{l^N}^{\otimes q}$ on
$\bar U\times \A^n$. In this way, we obtain spectral sequences
$E_1^{ab}, \bar E_1^{ab}$ and ${}'\bar E_1^{ab}$ as before, and a
cycle map
$$\bar{C}^q(U,n) \to \bar{E}_1^{-n,2q}$$
This induces the regulator map to \'etale cohomology
$$reg:CH^q(U,n) \to \bar{E}_2^{-n,2q} \to H^{2q-n}_{et}(\bar{U},\mu_{l^n}(q))$$ 
Moreover, the image lies in the $G_k$-invariant part.  Passing to the
limit yields a map
$$reg:CH^q(U,n)\to  H^{2q-n}_{et}(\bar{U},\Z_l(q))^{G_k}$$

Functorial properties of this regulator map follows in the same way as
the singular case.  Also as in singular case, this regulator map on
$CH^1(U,1)$ can be described explicitly.  From the Kummer sequence
$$ 1 \to \mu_{l^N} \to \OO^*_{\bar U} \stackrel {l^N} {\to}\OO^*_{\bar U} \to 1$$
the connecting homomorphism induces the map
$$\OO^*(U)/l^N\OO^*(U) \to H^1_{et}(\bar{U},\mu_{l^N})$$
Taking the inverse limit over $N$, we get the map
$$CH^1(U,1)\to CH^1(U,1)\otimes \Z_l \cong \varprojlim_N\OO^*(U)/l^N\OO^*(U)  \to H^1_{et}(\bar{U},\Z_l(1))$$
This is the same as the regulator map. The proof is similar to the
argument with the exponential sequence in the singular case.

\section{Invariant theory}

This section  refines some results from our earlier
paper \cite{AK}. We start by reviewing some notation from that paper.
The category of rational mixed Hodge structures forms a neutral
Tannakian category over $\Q$.  Let $\langle H\rangle$ denote the
Tannakian category generated by $H$. This is the full subcategory
consisting of all subquotients of tensor powers $T^{m,n}H=H^{\otimes
  m}\otimes (H^*)^{\otimes n}$. This construction extends to any set
of Hodge structures.  The Mumford-Tate group $MT(H)$ is the group of
tensor automorphisms of the forgetful functor from $\langle H\rangle$
to $\Q$-vector spaces.  By Tannaka duality $\langle H\rangle$ is
equivalent to the category of representations of this group.  When $H$
is a pure Hodge structure, $MT(H)$ can be defined in a more elementary
fashion as the smallest $\Q$-algebraic group whose real points
contains the image of the torus defining the Hodge structure.  We
define two auxiliary groups. The extended Mumford-Tate group $EMT(H)$
is $MT(\langle H, \QQ(1)\rangle)$ (some authors consider this to be
the Mumford-Tate group). The special Mumford-Tate group $SMT(H) =
\ker[EMT(H)\to \G_m]$ with respect to the map that is induced by the
inclusion $\langle \Q(1)\rangle\subset \langle H, \QQ(1)\rangle$.

Let $H$ be the first cohomology of a smooth quasi projective 
variety. 
% or more generally a mixed Hodge structure of type $\{(1,0),(0,1),(1,1)\}$.  
Then we have a description of $SMT(H)$ 
given by

\begin{lemma}\label{lemma:SMT}
  As a subgroup of $GL(H) = GL(V\oplus W)$
    $$ SMT(U)=
    \{\begin{pmatrix}
      I & 0 \\
      f & S
    \end{pmatrix} \mid S\in SMT(W) \text{ and $f\in \Phi $}\}.$$
  \end{lemma}

\begin{proof}
  See \cite{AK}.
\end{proof}

We want to refine this slightly.  We define three subspaces
$V_i\subset H$. Let $V_3 = W_1H$, and let $V_1 \subseteq H^{SMT(H)}$
be a complement to $V_3$ in $ W_1H+H^{SMT(H)} $, and finally choose
$V_2$ to be a complement to $V_1+V_3$ in $H$. Thus we have a
decomposition
\begin{equation}
  \label{eq:decompH}
  H = V_1\oplus V_2\oplus V_3  
\end{equation}
into vector spaces. Under this decomposition, $W_1H$ maps to $V_3$.
For the arguments below, it is convenient to assign weights to elements
of a mixed Hodge structure. Say that $x$ has weight $k$ if $x\in W_k$
and $x\notin W_{k-1}$. With this convention, we see that elements of $V_1$ and
$V_2$ are of weight 2 and elements of $V_3$ are of weight 1.

Let $G_c=SMT(V_3)$.
With respect to \eqref{eq:decompH}, $SMT(H)$ is a subgroup of the following matrix
group:
$$
\{\begin{pmatrix}
  I& 0 & 0\\
  0 &I& 0\\
  0 & f & S
\end{pmatrix} \mid S\in G_c\text{ and $f\in Hom(V_2,V_3)$}\}.$$
The unipotent radical $U(SMT(H))$ lies in the subgroup
\begin{equation}
  \label{eq:USMT}
  \{\begin{pmatrix}
    I& 0 & 0\\
    0 &I& 0\\
    0 & f & I 
  \end{pmatrix} \mid f\in Hom(V_2,V_3)\}.
\end{equation}

\begin{lemma}\label{lemma:Ureduction}
  For any nonzero $u\in V_2$, we can find a $g\in U(SMT(H)) $ such
  that $gu\not=u$, or equivalently such that $f(u)\not=0$ with respect
  to the matrix (\ref{eq:USMT}).
\end{lemma}

\begin{proof}
  See \cite{AK}.
\end{proof}

Let $BH^{n,s}(H)$ denote the space of Beilinson-Hodge cycles of weight $2s$ in
$H^{\otimes n}$. More precisely,
\begin{eqnarray*}
  BH^{n,s}(H) &=&Hom_{MHS}(\Q(-s),H^{\otimes n})  \\
  &\subseteq&Hom_{MHS}(\Q(-s),(H^{split})^{\otimes n}) 
\end{eqnarray*}

Note that if $n> 2s$ then $BH^{n,s}(H)=0$ hence the results stated below are 
trivially true. So we shall also assume $n\le 2s$.
To simplify book keeping, we will usually write tuples $(j_1,\ldots
j_n)$ as strings $j_1\ldots j_n$.  In particular, juxtaposition is
used to denote concatenation of strings, with exponents used for
repetition. For example, $1^2\,2\,3^0= 1\,1\,2$.
% This is combined with standard set theoretc notation in a fairly
% obvious manner. For example,
% $$\{1,3\}^22\{2,3\}= \{1,3\}\times \{1,3\}\times \{2\}\times \{2,3\}=
% \{1122, 1322, 3122, 3322, 1123,\ldots\}$$
Note that
\begin{equation}
  \label{eq:decompHn}
  H^{\otimes n}=\bigoplus_{j_1,\ldots, j_n} V(j_1\ldots j_n),
\end{equation}
where
$$ V(j_1\ldots j_n)=V_{j_1}\otimes\ldots\otimes V_{j_n}$$

Define $|j_1j_2 \ldots j_n|_3 = \#\{i:j_i=3\}$.

\begin{lemma}\label{lemma:weights}
The weight of a nonzero element of   $V(j_1\ldots j_n)$ is $2n-|j_1\ldots j_n|_3$.
\end{lemma}

\begin{proof}
  This follows trivially from the fact that elements of $V_1$ and
$V_2$ are of weight 2 and elements of $V_3$ are of weight 1. 
\end{proof}

Define the
space of elementary Hodge cycles $ EH^{n,s}(H)\subseteq H^{\otimes n}$ as
the subspace generated by products of elements of $V_1$ and
$V_3^{G_c}$ of weight $2s$. Thus from lemma \ref{lemma:weights},
\begin{equation}\label{eq:decompBHtriv}
  EH^{n,s}(H)=\bigoplus_{j_1,\ldots,j_n\in \{1,3\}, |j_1\ldots j_n|_3=2n-2s} 
  V(j_1\ldots j_n)^{G_c}
\end{equation}
where the action of $G_c$ on $V_1$ is trivial.  Note that $G_c$ can be
viewed as a quotient of $SMT(H)$ since $W$ is sub-Hodge structure of
$H$.

\begin{thm} \label{BH:main} $BH^{n,s}(H)=EH^{n,s}(H)$
\end{thm}

\begin{proof}
  Since $SMT(H)$ acts trivially on $V_1$ and the action of $SMT(H)$ on
  $V_3$ factors through $G_c$, the action of $SMT(H)$ on $EH^{n,s}(H)$
 is trivial. Also elements of $EH^{n,s}(H)$ are of weight $2s$,
  hence they are Beilinson-Hodge cycles of weight $2s$ in $H^{\otimes
    n}$.  This proves $EH^{n,s}(H) \subseteq BH^{n,s}(H)$.

  Let $\tau \in BH^{n,s}(H)$, our goal is to show that
  $\tau\in EH^{n,s}(H)$.  Let us decompose
  \begin{equation}
    \label{eq:tau}
\tau =\sum \tau_{j_1\ldots j_n}    
  \end{equation}
with respect to \eqref{eq:decompHn}.  Let $\tau^{EH}$ be the sum of
those terms $\tau_{j_1\ldots j_n}$ which are in $EH^{n,s}(H)$. We replace $\tau$ by
$\tau-\tau^{EH}$. Then  it is enough to show that $\tau$ equals
$0$. Suppose that it is nonzero. Since $\tau\in W_{2s}H^{\otimes n}$, by
lemma \ref{lemma:weights} we have that $|k_1\ldots
k_n|_3\ge 2n-2s$ for every nonzero term  $\tau_{k_1\ldots k_n}$ of
\eqref{eq:tau}. Also since $\tau$ is nonzero of weight $2s$, 
$\tau$ must have a nonzero term $\tau'=\tau_{j_1\ldots j_n}$ so that $|j_1\ldots j_n|_3=
2n-2s$. Without loss of generality,
suppose that $j_1=j_2=\ldots =j_{2n-2s}=3$ and the remainder of the $j_i$'s
are either $1$ or $2$. When at least one of these is $2$, then we will derive
a contradiction. We have $j_1\ldots
j_n=3^{2n-2s}1^{n_1}2^{n_2}1^{n_3}\ldots$ with $n_2>0$.  So using \eqref{eq:USMT}
and the lemma~\ref{lemma:Ureduction}, we can obtain a $g\in U(SMT(H))$
so that $g\tau'-\tau'$ has a nonzero component in
$V(3^{2n-2s\,}1^{n_1\,}3\,2^{n_2-1}1^{n_3}\ldots)$.  Also if $\tau''$
is any other term in $\tau$ then its image under $g-I$ does not lie in
$V(3^{2n-2s\,}1^{n_1\,}3\,2^{n_2-1}1^{n_3}\ldots)$. This contradicts
the invariance of $\tau$ under the action of $SMT(H)$. So $j_1\ldots
j_n$ must be of the form $3^{2n-2s}1^{2s-n}$.

By assumption $\tau'\notin EH^{n,s}(H)$. So there
exists $h\in G_c$ so that $(h-I)\tau' \ne 0$. Let $g$ be a lift
of $h$ to $SMT(H)$. We have $0\ne (g-I)\tau' \in
V(3^{2n-2s}1^{2s-n})$. So again for $\tau$ to be invariant under the
$SMT(H)$ action, there must exist another term $\tau''=\tau_{k_1\ldots
  k_n}$ in \eqref{eq:tau} such that $(g-I)\tau'' \in
V(3^{2n-2s}1^{2s-n})$. For this to happen, $k_i=2$ for some $1\le i
\le 2n-2s$ and $k_i=1$ for all $i > 2n-2s$. This would imply that
$|k_1\ldots k_n|_3 < 2n-2s$
contradicting the previous inequality. This completes the proof.
\end{proof}

\subsection{Beilinson-Tate case}
Let $k$ be a finitely generated field over $\Q$.  In analogy with the
constructions of the Mumford-Tate group above, we define $MG(H)$ to be
the image of the map $G_k\to GL(H)$ for a $G_k$-module $H$.  Let
$EMG(H)=MG(H\oplus\Q_l(1))$.  This group acts on both factors $H$ and
$\Q_l(1)$. The second action determines a homomorphism $EMG(H)\to
\G_m$.  Let $SMG(H)= \ker[EMG(H)\to \G_m]$.

For a smooth variety $U$ over $k$, let $MG(U)$ and $SMG(U)$ denote
$MG(H^1_{et}(\bar{U},\Q_l))$ and $SMG(H^1_{et}(\bar{U},\Q_l))$
respectively.
% If $U$ is also projective then The Mumford-Tate conjecture states
% that $Lie MT(H^1(\bar{U}))\times \Q_l= Lie MT_G(U)$.
Recall that $H_{et}^i(\bar{U}, \Q_l)$ carries an increasing filtration
$W_j$ called the weight filtration \cite{deligne-poids}.  The space
$W_j$ is defined as the sum of the (generalized) eigenspaces of a
geometric Frobenius $F_m$, at an unramified place $m$, with
eigenvalues having norms $(Norm_{k/\Q}(m))^{j'/2}$ for some $j'\le j$.
When $i=1$, this can be described explicitly as follows.  Choose a
smooth compactification $X$ of $U$ (which exists by
\cite{hironaka-resolution}).  Then
$$0=W_0\subseteq W_1 =im[H^1_{et}(\bar{X},\Q_l)\to H^1_{et}(\bar{U}, \Q_l)]\subseteq W_2 =H_{et}^1(\bar{U}, \Q_l)$$

\begin{lemma}\label{weight}
  $MG(U)$ preserves the the weight filtration on
  $H^1_{et}(\bar{U},\Q_l)$.
\end{lemma}

\begin{proof}
  This is evident from the above description.
\end{proof}

% $MG(U)$ also acts on $\Hom_{G_k}(\Q_l(1),H^1(U\times_k \C,\Q_l))$ by
% $(g\phi)(r)=g(\phi(r))$ where $g\in MG(U), \phi \in
% \Hom_{G_k}(\Q_l(1),H^1(U\times_k \C,\Q_l))$ and $r\in
% \Q_l(1)$. Since $\phi$ is $G_k$ equivariant
% $g(\phi(r))=\phi(g(r))=\phi(\alpha(g)(r))$ for some rational number
% $\alpha(g)$. Hence $g\phi=\alpha(g)\phi$.  Let $SMG(U)$ be the
% kernel of the group homomorphism from $MG(U) \to \G_m$ which sends
% $g \to \alpha(g)$.

%From now on, let $U$ be either a product of smooth curves or a
%semiabelian variety over $k$.  Fix a smooth compactification $X$ of
%$U$.  
Let $H= H_{et}^1(\bar{U},\Q_l)$ and let $W = W_1H=H_{et}^1(\bar
X,\Q_l)$.  Choose a (not necessarily $G_k$-invariant) complementary
subspace $V$ to $W$ in $H$. $SMG(H)$ being a subgroup of $MG(H)$, by
lemma \ref{weight}, it preserves the weight filtration on $H^1(U)$.
Therefore, with respect to the decomposition $H=V\oplus W$, we can
identify
$$SMG(H)\subseteq 
\{\begin{pmatrix}
  * & 0 \\
  * & *
\end{pmatrix}\}. $$ In particular,
$$\Phi = \ker[SMG(U)\to SMG(H/W\oplus W)]$$
% $\Phi$ the kernel of $SMG(U)\to SMG(H^1(X)\oplus (H/H^1(X)))$ and
% the unipotent radical of $SMG(U)$
is a subgroup of
$$\{\begin{pmatrix}
  I & 0 \\
  f & I
\end{pmatrix} \mid f\in Hom_{\Q_l}(V,W) \}.$$ In other words, $\Phi$
is a subspace of $\Hom_{\Q_l}(V, W)$.

\begin{cor}\label{cor:SMT}
  % As a subgroup of $GL(H) = GL(V\oplus W)$
  The group
    $$ SMG(H)=
    \{\begin{pmatrix}
      I & 0 \\
      f & S
    \end{pmatrix} \mid S\in SMG(W) \text{ and $f\in \Phi $}\}$$ Its
    Zariski closure
$$ \overline{SMG}(H)=
\{\begin{pmatrix}
  I & 0 \\
  f & S
\end{pmatrix} \mid S\in \overline{SMG}(W) \text{ and $f\in \Phi $}\}$$
The group $\Phi$ is the unipotent radical of $\overline{SMG}(H)$.
\end{cor}

  \begin{proof}
    The only thing to observe, for the first statement, is that
    $V\cong H/W\cong \Q_l(-1)^N$, so $SMG(H)$ acts trivially on it.
    The second follows from this.  Finally, by a theorem of Faltings
    \cite[Satz 3]{faltings} \cite[p 211]{fw} the action of $G_k$ on
    $W=H_{et}^1(\bar X,\Q_l)$ is semisimple. Therefore the Zariski
    closure of its image $MG(W)$ is reductive. $\overline{SMG}(W)$ is
    also reductive, because up to isogeny it is a direct factor of
    $\overline{MG}(W)$.
  \end{proof}

Since $SMT(H)$ and $\overline{SMG}(H)$ are similar in structure, the
results on invariants proved above for $SMT(H)$ holds for
$\overline{SMG}(H)$ as well.  In particular, we can decompose
$H=V_1\oplus V_2\oplus V_3$ so that $V_3=W$, $\overline{SMG}(H)$
acts trivially on $V_1$ and for every nonzero $u\in V_2$, there
exist an element $g\in \Phi$ such that $g u\ne u$.

As in Hodge case, let
\begin{eqnarray*}
  BT^{n,s}(H) &=&(H^{\otimes n}\otimes \Q_l(s))^{G_k}\otimes \Q_l(-s) \\
   &=&Hom_{MG(H)}(Q_l(s),H^{\otimes n})
\end{eqnarray*}

Let $G_c=\overline{SMG}(V_3)$. Define the space of elementary Tate-cycles 
$ET^{n,s}(H)\subseteq H^{\otimes n}$ as the subspace 
\begin{equation}
  ET^{n,s}(H)=\bigoplus_{j_1,\ldots,j_n\in \{1,3\}, |j_1\ldots j_n|_3=2n-2s} 
  V(j_1\ldots j_n)^{G_c}
\end{equation}
where the action of $G_c$ on $V_1$ is trivial. Again $G_c$ can
be viewed as a quotient of $SMG(H)$ since $W$ is $G_k$-submodule of
$H$.

\begin{thm} \label{BT:main} $BT^{n,s}(H)=ET^{n,s}(H)$
\end{thm}

The proof is same as the Hodge case.

\section{Beilinson's Hodge and Tate cycles on product of curves and semiabelian
  varieties}

For a variety $U$ over a field $k$, let $\bar{U}=U\times_k \bar{k}$. From the 
first section, we get maps
$$
\begin{cases}
  CH^i(U,j)\to Hom_{MHS}(\Z(0), H^{2i-j}(U,\Z(i))) &\text{if $U$ is over } \C\\
  CH^i(U,j)\otimes \Z_l\to H_{et}^{2i-j}(\bar{U},\Z_l(i))^{G_k} & \text{in general}
\end{cases}
$$

The first map is easily seen to be surjective for $i=j=1$ (see
\cite[Thm 5.13]{jannsen} or \cite[Thm 1.1]{AK}), and therefore the
Beilinson-Hodge conjecture holds integrally at this level. For a 
finitely generated field $k$, the same is true for Beilinson-Tate:

\begin{thm}(Jannsen, \cite[Thm 5.15]{jannsen})\label{11case} For any
  smooth variety $U$ over a finitely generated field $k$, the map
  $$reg: CH^1(U,1)\otimes \Z_l \to H^1(U,\Z_l(1))^{G_k}$$
  is surjective.
\end{thm}

In \cite{AK}, we defined
$$BH^q(U) = Hom_{MHS}(\Q(0), H^q(U,\Q(q)))$$ 
Similarly, we define the space of Beilinson-Tate cycles
$$BT^q(U) = H_{et}^{q}(\bar{U},\Q_l(q))^{G_k}$$
More generally, let
$$BH^{n,s}(U)=Hom_{MHS}(\Q(0), H^n(U,\Q(s)))$$
$$BT^{n,s}(U)=H_{et}^n(\bar{U},\Q_l(s))^{G_k}$$

The Beilinson-Hodge (respectively Beilinson-Tate) conjecture asserts
that the regulator maps from $CH^q(U,q)$ (respectively 
$CH^q(U,q)\otimes \Q_l$ ) surjects onto
$BH^q(U)$ (respectively $BT^q(U)$). As in the Beilinson-Hodge case,
note that the conjecture is only interesting for open varieties,
because it is vacuously true if the variety is projective.  In case of smooth projective 
varieties $BT^q(U)=0$ because $H_{et}^q(\bar{U},\Q_l)$ is pure of weight $2q$.

\begin{lemma}\label{lemma:key}
  If the products $BT^1(U)\times \ldots \times BT^1(U)\to BT^q(U)$ are
  surjective for all $q$, then the Beilinson-Hodge conjecture holds
  for $U$.
\end{lemma}

\begin{proof}
  This follows from the following commutative diagram and theorem
  \ref{11case}
  $$
  \xymatrix{
    CH^1(U,1)\times \ldots \times CH^1(U,1)\ar[r]\ar[d] & CH^q(U,q)\ar[d] \\
    BT^1(U)\times \ldots \times BT^1(U)\ar[r] & BT^q(U) }
  $$  
\end{proof}

\begin{cor}
  The Beilinson-Tate conjecture holds for a product of smooth curves
  over a finitely generated field of characteristic $0$.
\end{cor}

\begin{proof}
  Let $U = \prod U_i$, where $U_i$ are smooth curves. Let $H=
  H_{et}^1(U)$ and note that by  K\"unneth's formula and the theorem
  \ref{BT:main}
$$BT^q(U)=BT^{q,q}(U)=BT^{q,q}(H)\otimes \Q_l(q)=ET^{q,q}(H)\otimes\Q_l(q) =
  BT^1(U)^{\otimes q}.$$  
 So the hypothesis of lemma~\ref{lemma:key} holds.
\end{proof}

\begin{cor}\label{cor:BT}
  The Beilinson-Tate conjecture holds for a semiabelian variety
 over a finitely generated field of characteristic $0$.
\end{cor}

\begin{proof}
  Let $U$ be a semiabelian variety. Let $H = H_{et}^1(U)$. By the
  theorem \ref{BT:main}, we have that $BT^n(H) = BT^1(H)^{\otimes
    n}$. Now observe that $H^*(U) = \wedge^* H$ which is a direct
  summand of the tensor algebra. So the Beilinson-Tate cycles on
  $H^n(U)$ are given by products of Beilinson-Tate cycles on $H$.
\end{proof}

\subsection{Strong Beilinson's Hodge and Tate conjecture}

$k$ will stand for either $\C$ or a finitely generated field of
characteristic $0$.

\begin{prop}\label{strong}
Let $C_1,C_2,\ldots, C_n$ be smooth curves defined over $k$.
Let $U=C_1\times C_2\ldots C_n$ 
and let $X$  $X=\bar{C}_1\times\bar{C}_2\ldots \bar{C}_n$ where $\bar{C_i}$ denote 
the smooth compactification of $C_i$. 
  \begin{enumerate}
  \item If $k=\C$ and the Hodge conjecture is true
    for $X$ then $BH^{n,s}(U)$ is generated by algebraic cycles on
    $U$.
  \item If $k$ is finitely generated and the Tate conjecture is true for $X$ then
    $BT^{n,s}(U)$ is generated by algebraic cycles on $U$.
  \end{enumerate}
\end{prop}

\begin{proof}
  Let $H=H^1(U,\Q)$. By the K\"unneth formula, $H^*(U)$ can be written as a
  sum of tensor products of powers of $H$ with spaces generated by
  cycles. 
  By  theorem \ref{BH:main}, it is enough to prove that every cycle 
  in $EH^{n,s}(H)$ is algebraic. Moreover, it suffices to show that every 
  summand in \eqref{eq:decompBHtriv} is algebraic. Without loss of generality,
  we shall show that $V(3^{2n-2s}\,1^{2s-n})^{G_c}$ is algebraic.

  The Hodge conjecture for $X$ implies the surjectivity of the map
  $CH^{n-s}(X) \to BH^{2n-2s,n-s}(X)$.  We know that $H^{2n-2s}(X)$
  surjects onto $W_{2n-2s}H^{2n-2s}(U)$ \cite{deligne-hodge}, and this
  induces the surjection $BH^{2n-2s,n-s}(X)\onto
  BH^{2n-2s,n-s}(U)$. Hence we get the surjection $CH^{n-s}(U)\onto
  BH^{2n-2s,n-s}(U)=V(3^{2n-2s})^{G_c}\otimes \Q(n-s)$.

  By \cite{AK}, we have the surjection $$CH^{2s-n}(U,2s-n)\onto
  BH^{2s-n,2s-n}(U)=V(1^{2s-n})\otimes\Q(2s-n).$$

  Since regulators preserve products, we have the surjection from
  $$CH^s(U,2s-n)\onto V(3^{2n-2s})^{G_c}\otimes
  V(1^{2s-n})\otimes\Q(s)=V(3^{2n-2s}\,1^{2s-n})^{G_c}\otimes\Q(s)$$
  In view of theorem \ref{BT:main}, the proof of (2) is similar.
\end{proof}

The above proposition can be reduced to verifying the Hodge and the Tate conjecture for a
certain abelian variety. 

\begin{cor}\label{jacobian}
  The conclusion of the proposition is true if the Hodge (respectively Tate) conjecture holds
  for the product of Jacobians $J=J(\bar{C}_1)\times \ldots \times J(\bar{C}_n)$
\end{cor}

\begin{proof}
  Let $X=\bar{C}_1\times\bar{C}_2\times\ldots\bar{C}_n$.
  The Abel-Jacobi map induces a surjection $H^*(J)\to H^*(X)$.
  A Hodge (resp. Tate) cycle $\gamma$ on $X$ can be pulled back to a 
  Hodge (resp. Tate) cycle on  $J$, and then the corresponding algebraic 
  cycle can be pushed back down to $X$ to prove algebraicity of $\gamma$.
\end{proof}

\begin{prop}\label{main2}
Let $U$ be a semiabelian variety and $X$ be a smooth compactification
of $U$ both defined over $k$.
  \begin{enumerate}
  \item If $k=\C$ and the Hodge conjecture is true
    for $X$ then $BH^{n,s}(U)$ is generated by algebraic cycles on
    $U$.
  \item If  $k$ is finitely generated and the Tate conjecture is true for $X$ then
    $BT^{n,s}(U)$ is generated by algebraic cycles on $U$.
  \end{enumerate}
\end{prop}

\begin{proof}
Let $H=H^1(U,\Q)$.
Again note that $H^n(U,\Q)$ is a direct summand of $H^{\otimes n}$. So
$BH^{n,s}(U)$ is a direct summand of $BH^{n,s}(H)\otimes \Q(s)$.  
Theorem \ref{BH:main} yields that $BH^{n,s}(U)$ is a direct summand of 
$EH^{n,s}(H)\otimes \Q(s)$. So without loss of generality, it is enough 
to show that $BH^{n,s}(U)\cap [V(3^{2n-2s}\,1^{2s-n})^{G_c}\otimes Q(s)]$ is algebraic.
Again because of the splitting from wedge-products to tensor products, we know that 
$$BH^{n,s}(U)\cap [V(3^{2n-2s}\,1^{2s-n})^{G_c}\otimes 
\Q(s)]=BH^{2n-2s,n-s}(U)\otimes BH^{2s-n,2s-n}(U).$$
The rest of the proof is same as for product of curves.
\end{proof}

\begin{cor}\label{main3}
Let $A$ be an abelian variety, $T$ be a torus and $U$ be a semiabelian variety given 
by an extension of $A$ by $T$. The conclusion of the above proposition is true if
Hodge (resp. Tate) conjecture holds for $A$.
\end{cor}

\begin{proof}
After replacing $k$ by a  finite extension (which is harmless), we can
assume that $T$ is split. Then $U$ can  be compactified by 
the projective space bundle $X$ over $A$.
Then  $H^*(X)=H^*(A)\otimes H^*(\PP^r)$ where
$r$ is the rank of torus $T$. So Hodge (resp. Tate) conjecture for $X$ is true if and only
if Hodge (resp. Tate) conjecture for $A$ is true.
\end{proof}

There are a number of well known criteria for the Hodge conjecture
for abelian varieties (cf. \cite{gordon}). Many of them are obtained by
analyzing the Mumford-Tate group of the first cohomology of abelian varieties.
For a smooth projective curve $C$ of nonzero genus, let $Lef(C)$ be the centralizer 
of $End(J(C))\otimes \Q$ in the symplectic group $Sp(H^1(C))$. The following is one 
of the basic criteria to decide Hodge conjecture for $J(C)^n$.

\begin{prop}[Murty]
  Let $C$ be a smooth projective curve such that $End(J(C))\otimes\Q$ is field and 
  $Lef(C)=SMT(C)$ then Hodge conjecture holds for $J(C)^n$ for all $n$.
\end{prop}

\begin{proof}
  This follows from \cite[Theorem 6.2]{gordon}.
\end{proof}
 
In fact for most smooth projective curves $C$ the Hodge conjecture holds for
$C^n$ as shown by the following proposition in \cite[Prop 6.5]{arapura}.

\begin{prop}
  There exists a countable union $S$ of proper Zariski closed sets in the moduli
  space $M_g(\C)$ of curves of genus $g\ge 2$, such that if $C\in M_g(\C)\setminus S$ 
  then the generalized Hodge conjecture holds for all powers of its Jacobian $J(C)$.
\end{prop}

Any such $C$ is called a very general curve. We combine some of these results to 
obtain the following corollary.

\begin{cor}
  Let $C$ be an open subset of a smooth projective complex curve $\bar{C}$. 
 The strong version of the Beilinson-Hodge 
  conjecture holds for $C^n$ if $\bar{C}$ is
  one of the following
  \begin{enumerate}
  \item a curve of genus 1, 2 or 3.
  \item a curve of prime genus such that its Jacobian is simple.
  \item a Fermat curve $x^m+y^m+z^m=0$ with $m$ prime or less than 21.
  \item a curve admitting a surjection from a modular curve $X_1(N)$.
  \item a very general curve.
  \end{enumerate}
\end{cor} 

\begin{proof}
  In view of corollary \ref{jacobian}, it is enough to show that Hodge conjecture
  holds for $J(C)^n$ in all the cases:
  \begin{enumerate}
  \item When $C$ is an elliptic curves, Hodge conjecture for $C^n$ was first proved 
    by Tate but never published his proof. In \cite{murasaki} Murasaki showed that 
    Hodge cycles on $C^n$ are generated from divisors. 
    For curves of genus 2 and 3 it was worked out by Mumford in an 
    unpublished work but a proof can be found in \cite{moonen-zarhin}.
  \item
    For these curves the result was shown by Tankeev and Ribet (\cite[p. 525]{ribet}).
  \item
    Shioda (\cite[Theorem IV]{shioda}) treated the Fermat curves.
  \item
    This is given by work of Hazama and Murty (see \cite{hazama}).
  \item
    This follows from the above proposition.
  \end{enumerate}
\end{proof}

\begin{cor}
  Let $U$ be a semiabelian variety obtained by an extension of the abelian variety
  $A$ by a torus. Strong versions of Beilinson-Hodge and Beilinson-Tate conjectures
  hold for $U$ if $A$ is defined over a number field and is one of the following type:
  \begin{enumerate}
    \item dim $A$ is 2 or an odd number and $\End^0(A\times_k \bar{k})=\Q$.
    \item dim $A=2d$ where $d$ is an odd number and $\End^0(A\times_k \bar{k})$ is a
      real quadratic field or an indefinite quaternion algebra over $\Q$.
    \item dim $A=4d$ where $d$ is an odd number and $\End^0(A\times_k \bar{k})$ is an
      indefinite quaternion algebra over a real quadratic field.
  \end{enumerate}
\end{cor}

\begin{proof}
 In view of corollary \ref{main3}, it is suffices to check that Hodge and Tate 
conjectures are true for $A$ in the three cases. The first case was proved by 
Serre in an unpublished notes. Serre's methods were extended by W. Chi to show
Mumford-Tate conjecture, Hodge conjecture and Tate conjecture in all the three
cases (\cite[Theorem 8.5, 8.6, 8.8]{chi}).
\end{proof}

Tankeev (\cite{tankeev}) has extended these results of Serre and Chi to some other 
classes of abelian varieties. So corollary \ref{main3} can be applied in these 
cases as well.
% \begin{cor}
%   Let $U$ be a smooth variety over a number field, dominated by the product 
%   of open subsets of a Fermat curve then the strong version Beilinson-Tate 
%   conjecture holds for $U$.
% \end{cor}

% \begin{proof}
%   Shioda has shown (\cite{shioda-ht}) that the Tate conjecture 
% \end{proof}

\section{Case of smooth varieties dominated by products of curves}

In this section we will deduce Beilinson-Hodge and Beilinson-Tate
conjectures for smooth varieties $U$ which are dominated by product of
curves by a proper surjective morphism.

% For a smooth complex variety $U$, let $BH^q(U)=Hom_{MHS}(\Q(-q),
% H^q(U,\Q))$.

\begin{lemma}
  Suppose that the Beilinson-Hodge (respectively Beilinson-Tate)
  conjecture holds for a smooth variety $Y$ and there is a proper
  surjective map $p:Y \to X$ with $X$ smooth. Then the Beilinson-Hodge 
  (respectively Beilinson-Tate) conjecture holds for $X$.
\end{lemma}

\begin{proof}
  We treat the case of Beilinson-Hodge only, since the Beilinson-Tate 
  case is identical. Let $r=\dim(Y)-\dim(X)$ be the relative dimension.
  % Consider the diagram
  The properties of the regulator map in section \ref{cycle-map} shows
  the commutativity of the following diagram.
  $$
  \xymatrix{
    CH^i(Y, j)\otimes \Q\ar[r]^{reg_Y}\ar[d]^{p_*} & BH^{i,2i-j}(Y)\ar[d]^{p_*} \\
    CH^{i-r}(X, j)\otimes \Q\ar[r]^{reg_X} & BH^{i-r,2i-j-2r}(X) }
  $$
  By assumption, $reg_Y$ is surjective. Let $f: Z\into Y$ be a subvariety 
  so that the restriction $g: Z\to X$ is generically finite and surjective.
  Then $g_*: H^*(Z)\to H^*(X)$ is surjective by the projection formula.
  Since $g=p\circ f$, $p_*$ on the right in the above commutative 
  diagram is surjective. Therefore $reg_X$ must be surjective.
\end{proof}

\begin{cor} \label{maincor} 
  Let $U$ be a product of smooth curves over $k$, $V$ be a smooth variety 
  over $k$ and $p:U\to V$ be a proper surjective morphism. Let $X$ be a 
  smooth compactification of $U$.
  \begin{enumerate}
    \item If $k=\C$ then the Beilinson-Hodge conjecture holds for $V$ and if 
      Hodge conjecture holds for $X$ then strong version of Beilinson Hodge
      conjecture also holds for $V$.
    \item If $k$ is a number field then the Beilinson-Tate conjecture holds 
      for $V$ and if Tate conjecture holds for $X$ then the strong version of
      Beilinson Tate conjecture also holds for $V$.
  \end{enumerate}
\end{cor}

% \begin{proof}
%   Let $H=H^1(U,\Q)$. By Abel's theorem, \cite[Theorem 1.1]{AK}, the
%   regulator map $reg:\cO(U)^*\to
%   Hom_{MHS}(\Q(-1),H^1(U,\Z))=BH^1(U)=BH^1(H)$ is surjective.  By
%   \cite[Theorem 3.1]{AK}, $BH^1(H)\times\ldots\times BH^1(H) \to
%   BH^q(H)$ is a surjection.  By the lemma \ref{directsummand},
%   $BH^q(U)$ is contained in $BH^q(H)$.  So the first part of the
%   corollary follows from \cite[Lemma 1.1]{AK}.

%   Similarly for the second part, Jannsen's theorem yields
%   surjectivity of $reg:CH^1(U,1)\otimes Q_l \to BT^1(U)$.  The rest
%   follows from theorem \ref{mainthmAK} and lemma
%   \ref{directsummand}.
% \end{proof}

% \subsection*{Complement of translates of theta divisor in the
%   Jacobian of a curve}
\begin{prop}
Every semiabelian variety $U$ is dominated by product of curves.
\end{prop}
\begin{proof}
Let $C_1,C_2,\ldots C_n$ be (possibly affine) curves in $U$ such that they 
generate $U$ and $X_i$ be a one-point-compactification of $C_i$. 
By \cite[Chapter V]{serre-jacobian} there exist commutative group schemes called generalized
Jacobians $J_i=J(X_i)$ and morphisms $\phi_i:C_i\to J_i$ which are universal 
for commutative group schemes. This defines a surjective morphism 
$\phi:J_1\times J_2\times\ldots J_n\to U$ given by 
$\phi(x_1,\ldots x_n)=\phi_1(x_1)\cdot\phi_1(x_2)\cdot\ldots \phi(x_n)$.
There is also a surjective morphism from 
$C_i^{\pi_i}\to J_i$ where $\pi_i$ is the arithmetic genus of $X_i$. Hence
$U$ is dominated by product of curves. 
\end{proof}

We can recover corollary \ref{cor:BT} from the above proposition.  We give another
example where corollary \ref{maincor} applies.  Let $C$ be a smooth projective curve of
genus $g$ and $P$ be a $k$-rational point on it. Let $J=J(C)$ be the
Jacobian of $C$. Viewing points on $J$ as degree zero divisors on $C$,
let $f_n:C^n\to J$ be the Abel-Jacobi morphism given by
$(P_1,\ldots, P_n)\mapsto P_1+\ldots +P_n - nP$. We can identify $C$
with its image $f_1(C)$.  By Jacobi inversion, $f_g$ is a generically
finite surjective morphism. The image of $f_{g-1}$ is the
$\Theta$-divisor on $J$ (\cite{abvar}).

\begin{cor}
  Given rational points $Q_1,\ldots Q_m\in C(k)$, the Beilinson-Hodge
  and Beilinson-Tate conjectures holds for $J\setminus \cup_i
  (\Theta+Q_i)$. Moreover if Hodge (resp. Tate) conjecture holds for $J$
  then the strong Beilinson-Hodge (resp. Beilinson-Tate) conjecture holds
  for $J\setminus \cup_i  (\Theta+Q_i)$.
\end{cor}

\begin{proof}
  The map $(C\setminus \{Q_1,\ldots Q_m\})^g\to J\setminus \cup_i
  (\Theta+Q_i)$ given by the restriction of $f_g$ is a proper
  surjective morphism.
\end{proof}

\end{document}